\documentclass[11pt]{amsart}

\pdfoutput=1

\usepackage[T1]{fontenc}
\usepackage{lmodern}

\usepackage{esint,amssymb} 
\usepackage{graphicx}
\usepackage{MnSymbol}
\usepackage{mathtools} 
\usepackage[colorlinks=true, pdfstartview=FitV, linkcolor=blue, citecolor=blue, urlcolor=blue,pagebackref=false]{hyperref}
\usepackage{orcidlink}
\usepackage{microtype}

\usepackage{bm}
\usepackage{dsfont}
\usepackage{mathrsfs}
\usepackage{xcolor}

\usepackage{caption}

\newtheorem{proposition}{Proposition}
\newtheorem{theorem}[proposition]{Theorem}

\newtheorem{corollary}[proposition]{Corollary}

\theoremstyle{remark}
\newtheorem{remark}[proposition]{Remark}

\theoremstyle{definition}

\numberwithin{equation}{section}
\numberwithin{proposition}{section}

\newcommand{\R}{\mathbb{R}}

\renewcommand{\le}{\leqslant}
\renewcommand{\ge}{\geqslant}
\renewcommand{\leq}{\leqslant}
\renewcommand{\geq}{\geqslant}

\renewcommand{\subset}{\subseteq}

\newcommand{\Ll}{\left}
\newcommand{\Rr}{\right}
\renewcommand{\d}{\mathrm{d}}
\newcommand{\dr}{\partial}

\newcommand{\mcl}{\mathcal}
\newcommand{\msf}{\mathsf}

\newcommand{\al}{\alpha}
\newcommand{\be}{\beta}
\newcommand{\ga}{\gamma}

\newcommand{\si}{\sigma}

\renewcommand{\v}{\mathbf{v}}
\renewcommand{\phi}{\varphi}

\begin{document}

\author[Jean-Christophe Mourrat]{Jean-Christophe Mourrat\,\orcidlink{0000-0002-2980-725X}}
\address[Jean-Christophe Mourrat]{Department of Mathematics, ENS Lyon and CNRS, Lyon, France}

\subjclass[2010]{}
\date{}

\title[PL conditions do not guarantee convergence of GDA]{PL conditions do not guarantee convergence of gradient descent-ascent dynamics}

\begin{abstract}
We give an example of a function satisfying a two-sided Polyak-\L{}ojasiewicz (PL) condition for which a gradient descent-ascent flow line fails to converge to the saddle point, circling around it instead. We can even impose the function to be strongly convex in one variable and to satisfy a PL condition in the other variable.
\end{abstract}

\maketitle

\begin{figure}[h]
  \centering
  \includegraphics[width=\linewidth]{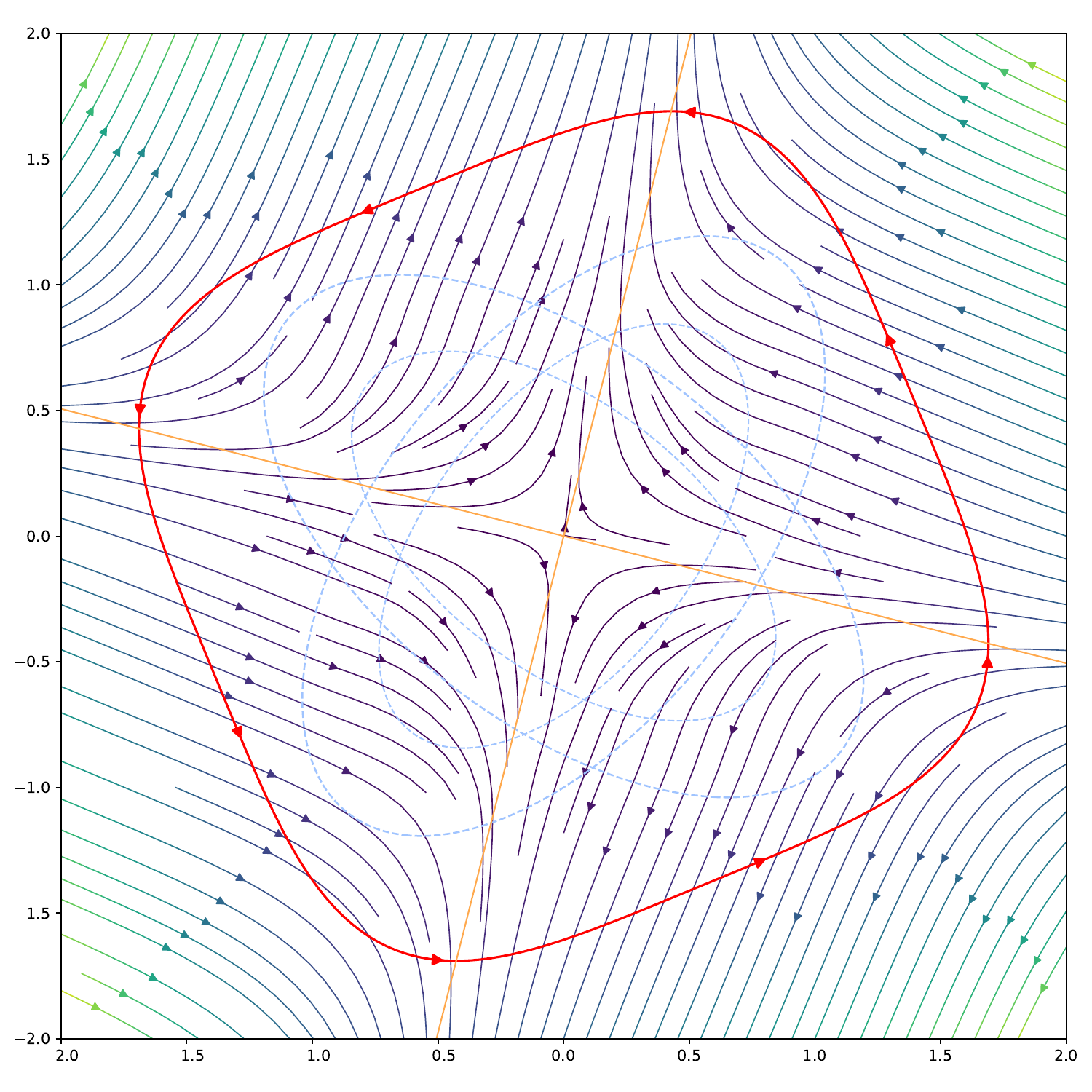}
  \captionsetup{width=0.95\linewidth}
  \caption{{\small The flow lines with a color scale from dark blue to yellow are the level lines of the function $f$ we build for Theorem~\ref{t.main} (the color scale indicates the magnitude of $\v$). The value of $f$ is not shown and is prescribed along the two orange lines according to \eqref{e.f.on.mclX}. These two lines are also the set of points at which the level line of $f$ is horizontal or vertical (i.e.,\ where $\dr_x f = 0$ or $\dr_y f = 0$). The red trajectory is a gradient descent-ascent flow line on $f$. The ellipses are the level lines of the quadratic forms inside the two occurrences of the function $\phi$ in \eqref{e.def.v}, for the values $1/2$ and $1$. In particular, the level lines of $f$ are tangent to the vector field given in \eqref{e.def.w} on the region contained by every ellipse, and to the vector field given in \eqref{e.v.outside} on the region that is outside of every ellipse. 
}}
  \label{f.vector_field}
\end{figure}

%
%
%
%
%
%

\section{Introduction}

A differentiable function $f : \R^d \to \R$ that is bounded from below is said to satisfy a Polyak-\L{}ojasiewicz (PL) condition \cite{lojasiewicz1963topological, polyak1963gradient} if there exists $C < +\infty$ such that for all $z \in \R^d$,
\begin{equation}  \label{e.gen.PL}
f(z) - \inf f \le C |\nabla f(z)|^2.
\end{equation}
Under this condition, all gradient descent trajectories converge to a minimizer of $f$ at linear speed \cite{karimi2016linear}. 
A natural question is whether analogous conditions can guarantee convergence for saddle-point problems. 
Min-max optimization problems of the form
\begin{equation*}
\min_{x \in \mcl X} \max_{y \in \mcl Y} f(x,y)
\end{equation*}
arise in a variety of contexts, including robust optimization \cite{ben2009robust} and generative adversarial networks \cite{goodfellow2014generative}. For definiteness, we set $\mcl X$ and $\mcl Y$ to be Euclidean spaces. A natural procedure for finding a saddle point is to iteratively take small steps in the direction of $(-\nabla_x f, \nabla_y f)$. In the regime of infinitesimally small steps, this amounts to studying the gradient descent-ascent (GDA) flow $(z(t))_{t \ge 0}$ solving
\begin{equation}
\label{e.def.grad.asc.desc}
\dr_t z(t) = 
\begin{pmatrix}
-\nabla_x f \\
\nabla_y f
\end{pmatrix}
(z(t)).
\end{equation}
If the functions $(f(\cdot,y))_{y \in \mcl Y}$ are uniformly strongly convex, and the functions $(f(x,\cdot))_{x \in \mcl X}$ are uniformly strongly concave, then one can show that the GDA flow converges to a saddle point of $f$ \cite{demyanov1972numerical}. If $f$ is only convex-concave (without the ``strongly'' qualifier), then there are counter-examples to convergence of the GDA flow, such as with the function $f : (x,y) \mapsto xy$ for which the flow circles around the origin. Yet under the sole convexity-concavity condition, other first-order algorithms such as the extragradient method do succeed in finding a saddle point for $f$ (assuming that one exists) \cite{korpelevich1976extragradient}. 

There has been significant effort to weaken the strong convexity-concavity assumption, for instance by imposing strong convexity in only one variable, or by replacing strong convexity-concavity with PL conditions \cite{doan2022convergence, lin2020gradient, lu2020hybrid, nouiehed2019solving, yang2020global}. We say that $f$ satisfies a \emph{two-sided PL condition} if the functions $(f(\cdot,y))_{y \in \mcl Y}$ and $(-f(x,\cdot))_{x \in \mcl X}$ satisfy a PL condition with a uniform constant. Under this condition, modified versions of the GDA flow were shown to converge to a saddle point in~\cite{doan2022convergence, yang2020global}. The modifications crucially impose a sufficiently large separation of timescales between the evolutions of the variables $x$ and $y$; see also \cite{lin2020gradient, lu2020hybrid, nouiehed2019solving} for related two-timescale approaches. Another positive result under the two-sided PL condition is that the GDA flow~\eqref{e.def.grad.asc.desc} itself converges to a saddle point if initialized sufficiently close to it (see Proposition~\ref{p.loc.conv} below). 

In view of these results, one may expect that the two-sided PL condition in fact guarantees global convergence of the GDA flow to a saddle point (assuming one exists). The point of this paper is to show that this is not so.

\begin{theorem}
\label{t.main}
There exists a $\mcl C^\infty$ function $f : [-1,1]^2 \to \R$ with a unique critical point at the origin, and a constant $C < +\infty$ such that for every $x, y \in [-1,1]^2$,
\begin{equation}  
\label{e.2PL}
\begin{cases}
\displaystyle{f(x,y) - \inf_{x' \in [-1,1]} f(x',y) \le C |\dr_x f(x,y)|^2},\\
\displaystyle{\sup_{y' \in [-1,1]} f(x,y') - f(x,y) \le C |\dr_y f(x,y)|^2},
\end{cases}
\end{equation}
and yet, for every $z(0)$ in some open subset of $[-1,1]^2$, the GDA flow given by \eqref{e.def.grad.asc.desc} is periodic.
\end{theorem}
Building upon Theorem~\ref{t.main}, we can then construct a suitable modification of the function $f$ for which the same conclusion as in Theorem~\ref{t.main} holds, but with the assumption of the two-sided PL condition strengthened into a strongly-convex/PL condition.
\begin{corollary}
\label{c.convPL}
There exists a $\mcl C^\infty$ function $F : [-1,1]^2 \to \R$ with a unique critical point at the origin, and a constant $C < +\infty$ such that for every $x, y \in [-1,1]^2$,
\begin{equation}  
\label{e.convPL}
\begin{cases}
\displaystyle{\partial_x^2 F(x,y) \ge C^{-1}},\\
\displaystyle{\sup_{y' \in [-1,1]} F(x,y') - F(x,y) \le C |\dr_y F(x,y)|^2},
\end{cases}
\end{equation}
and yet, for every $z(0)$ in some open subset of $[-1,1]^2$, the GDA flow given by \eqref{e.def.grad.asc.desc} (with $f$ replaced by $F$) is periodic.
\end{corollary}
It is immediate to verify that if a function $f$ satisfies \eqref{e.2PL} and admits a critical point, say at the origin (i.e.\ $\nabla f(0,0) = 0$), then this critical point is a saddle point, or more precisely, for every $x,y \in [-1,1]$, we have
\begin{equation*}  
f(0,y) \le f(0,0) \le f(x,0).
\end{equation*}
The function $f$ that we build to show Theorem~\ref{t.main} is displayed on Figure~\ref{f.vector_field} (up to a rescaling of the variables to bring them back to $[-1,1]^2$). For every $(x,y)$ in a neighborhood of the origin, this function is given by $f(x,y) = \frac \ga 2 x^2 + xy - \frac \ga 2 y^2$, for some $\ga \simeq 0.2531$. It is straightforward to check that this implies the convergence of the GDA flow towards the saddle point if one initializes the flow sufficiently close to the origin. This is not specific to this example, as we clarify in Proposition~\ref{p.loc.conv} below.

As we move away from the origin, we will progressively deform the function~$f$ so that the GDA flow then admits an integral of motion (i.e.\ a quantity that is preserved along the flow). This quantity is the $L^4$ norm after a rotation by $\pi/8$ (see also the red trajectory on Figure~\ref{f.vector_field}). 

In the statement of Theorem~\ref{t.main}, it is possible to replace the compact domain $[-1,1]^2$ by $\R^2$ if one so wishes. In order to do so, it indeed suffices to undo the said deformation, so that outside of a sufficiently large bounded region, the function $f$ becomes again the quadratic form that we set it to be near the origin. 

When the function $f$ is strongly convex-concave, the operator $V := (\nabla_x f, -\nabla_y f)$ is strongly monotone \cite{rockafellar1970monotone}, and in this case, the global
convergence of the GDA flow follows at once from the identity
$\frac 1 2 \dr_t |z - z^*|^2 = -(V(z) - V(z^*)) \cdot (z - z^*)$, where $z^*$ denotes the saddle point. In contrast, the two-sided PL condition does not imply that $V$ is monotone. The counter-example presented here is thus consistent with the understanding that this convergence is tied to the strong monotonicity of $V$. 

As announced, we also show for reference that the two-sided PL condition in~\eqref{e.2PL} does guarantee local convergence to a saddle point. For every $r > 0$, we denote by $B_r$ the open Euclidean ball of radius $r$ centered at the origin. 
\begin{proposition}[local convergence]
\label{p.loc.conv}
Let $f : [-1,1]^2 \to \R$ be a $\mcl C^2$ function that has a critical point at the origin and is such that neither $f(0,\cdot)$ nor $f(\cdot,0)$ is a constant function, and let $C < +\infty$ be such that \eqref{e.2PL} holds for every $x,y \in [-1,1]^2$. There exists $r > 0$ such that for every $z(0) \in B_r$, the GDA flow given by \eqref{e.def.grad.asc.desc} converges to the origin. 
\end{proposition}

We stress that the results presented here only concern continuous-time GDA dynamics. I expect that these results remain valid for discrete GDA algorithms, but a rigorous proof of this is left for future work. To do this, it would likely be convenient to start by slightly deforming the example presented here so that the continuous dynamics has an attractive cycle, in place of a continuous family of marginally stable cycles as we have here. 

We briefly mention a connection between the questions discussed here and recent works concerning the functional 
\begin{equation*}  
F : \Ll\{
\begin{array}{rcl}  
\mcl P_{H}(\mathbb T^d) \times \mcl P_{H}(\mathbb T^d) & \to & \R \\
(\mu,\nu) & \mapsto & \int g(x,y) \, \d \mu(x) \, \d \nu(y) + H(\mu) - H(\nu),
\end{array}
\Rr.
\end{equation*}
where $H(\mu) := \int_{\mathbb T^d} \log \Ll( \frac{\d \mu}{\d x} \Rr) \, \d \mu$ is the entropy of $\mu$, $\mcl P_{H}(\mathbb T^d)$ is the space of probability measures on the $d$-dimensional torus $\mathbb T^d$ with finite entropy, and $g : \mathbb T^d \times \mathbb T^d \to \R$ is a continuous function. In \cite{domingo2020mean, wang2024open}, the authors ask whether the Wasserstein GDA flow on $F$ is convergent. Depending on the function $g$, the function $F$ may not necessarily be convex-concave for the Wasserstein geometry. By \cite{otto2000generalization}, a two-sided PL assumption on $F$ corresponds to assuming uniform log-Sobolev inequalities for the minimizers of $(F(\cdot, \nu))_{\nu \in \mcl P(\mathbb T^d)}$ and $(-F(\mu,\cdot))_{\mu \in \mcl P(\mathbb T^d)}$; this is automatic on $\mathbb T^d$ and would be a natural assumption if the torus $\mathbb T^d$ was replaced by the full space $\R^d$. In this infinite-dimensional context, local convergence results have been obtained in \cite{seo2026local, wang2026local}, and two-timescale approaches have been considered in \cite{lu2023two, ma2022provably}. The question of \cite{domingo2020mean, wang2024open} was very recently answered in the negative in \cite{mourrat2026oscillating}.

The rest of the paper is organized as follows. In Section~\ref{s.gen}, we give simple sufficient conditions for a function defined on a compact interval of $\R$ to be PL, and for a function defined on a compact subset of $\R^2$ to satisfy the two-sided PL condition. We also prove Proposition~\ref{p.loc.conv}. We then proceed to construct a function $f$ that satisfies the requirements of Theorem~\ref{t.main} in Section~\ref{s.constr}. As discussed in the caption to Figure~\ref{f.vector_field}, we first build the level lines of $f$, which we identify as the flow lines of an explicit vector field. One key aspect of the construction is that the set of points where the level line is horizontal (i.e.\ the set of points where $\dr_x f = 0$) intersects any horizontal line only once (and similarly for the set of points of vertical tangency). Roughly speaking, the point of Section~\ref{s.gen} is to show that this is essentially sufficient to guarantee that the function $f$ satisfies the two-sided PL condition; the only additional information that we need is that $\dr_x^2 f$ does not vanish wherever $\dr_x f = 0$, and similarly in the $y$ direction. To complete the construction of the function $f$, we need to specify a value for each of the level lines. To facilitate this construction and also the verification of the two-sided PL condition, we do so by specifying the value of $f$ on the set of points where $\dr_x f$ or $\dr_y f$ vanishes, which by our construction is a union of two lines (see the orange lines on Figure~\ref{f.vector_field}). Finally, we give a proof of Corollary~\ref{c.convPL} in Section~\ref{s.convPL}.


%
%
%
%
%
%

\section{General properties of PL functions}
\label{s.gen}

Let $I$ be a compact interval of $\R$. 
We say that a differentiable function $f : I \to \R$ satisfies the PL condition if there exists a constant $C < +\infty$ such that for every $x \in I$, we have
\begin{equation}  
\label{e.PL}
f(x) - \inf f \le C |f'(x)|^2.
\end{equation}

\begin{proposition}[Criterion for PL condition]
\label{p.simple}
Let $f : I \to \R$ be a $\mcl C^2$ function that is not constant. The function $f$ satisfies the PL condition if and only if for every $x \in I$, we have the implication
\begin{equation}  
\label{e.simple}
f'(x) = 0 \quad \implies \quad f''(x) > 0.
\end{equation}
\end{proposition}
\begin{remark}
Notice that Proposition~\ref{p.simple} does not assume that the minimizer of $f$ is unique; this uniqueness follows from the proof below under any of the two stated equivalent conditions. 
\end{remark}
\begin{proof}[Proof of Proposition~\ref{p.simple}]
We write $I = [a,b]$, and first show the direct implication, assuming that $f$ satisfies the PL condition. Without loss of generality, we assume that $\inf f = 0$. We see from~\eqref{e.PL} that if we have $f'(x_0) = 0$ for some $x_0 \in I$, then we must have $f(x_0) = \inf f = 0$. Let $J$ be a connected component of the open set $\{f > 0\}$. The set $J$ is well defined since we assume that $f$ is not constant. If this set is the entire interval $I$, then there is nothing to show. Else, one of the endpoints of $J$ is different from $\{a,b\}$. For definiteness let us assume that the left endpoint of $J$, say $x_0$, is not $a$. In this case we have $f(x_0) = 0$. We also note that $f'$ cannot vanish in $J$ (since this would imply the vanishing of $f$ as previously observed), and since $f \ge 0$, we must have that $f' > 0$ on $J$. By~\eqref{e.PL}, we deduce that for every $x \in J$,
\begin{equation*}  
\frac{f'(x)}{\sqrt{f(x)}} \ge\frac{1}{\sqrt{C}}. 
\end{equation*}
Integrating this, we obtain that for every $x \in J$, 
\begin{equation*}  
\sqrt{f(x)} \ge \frac{x-x_0}{2\sqrt{C}}. 
\end{equation*}
This implies in particular that $f''(x_0) \ge \frac 1 {2C}$. We have seen that $f' > 0$ in~$J$, so it is not possible that $f$ vanishes at the right endpoint of $J$; in other words, we must have $J = (x_0,b]$. Since $f''(x_0) > 0$, the same picture holds as well to the left of $x_0$: we must have $f' < 0$ everywhere on $[a,x_0)$. The point~$x_0$ is thus the unique point at which $f'$ vanishes, and we have verified that $f''(x_0) > 0$, so the proof of the direct implication is complete.

We now turn to the converse implication, assuming that for every $x \in I$, the implication \eqref{e.simple} is valid. 
Under the stated assumption, if $f'(x) = 0$ for some $x \in I$, then in a neighborhood of~$x$, we have that $f'$ is negative to the left of $x$ and is positive to the right of $x$. This implies that $f'$ can vanish at most once. 
If $f'$ does not vanish anywhere, then by continuity of $f'$ and compactness we have that $\msf c := \inf |f'|^2 > 0$, so it suffices to take $C = (\sup f - \inf f)/\msf c$ to ensure that \eqref{e.PL} holds. Now let us assume that $f'$ vanishes at some $x_0 \in I$. We can find a neighborhood $U$ of $x_0$ such that $f'' \ge f''(x_0)/2 =: c > 0$ on $U$. Using also that $f''$ is bounded, we can integrate this to
\begin{equation*}  
\text{for every } x \in U, \quad f(x) - f(x_0) \le \frac {\sup f''} 2 (x-x_0)^2 \ \text{ and } \ |f'(x)|\ge c|x-x_0|. 
\end{equation*}
The PL condition \eqref{e.PL} is thus satisfied for every $x \in U$ provided that we choose $C \ge \sup f''/(2 c^2)$.
Since $|f'|^2$ is bounded away from zero in $I \setminus U$, we can again make sure that \eqref{e.PL} is satisfied by enlarging $C < +\infty$ as necessary.
\end{proof}

For a function $f : I^2 \to \R$, we say that $f$ satisfies the two-sided PL condition if there exists a constant $C < +\infty$ such that \eqref{e.2PL} is satisfied for every $x,y \in I^2$, with the interval $[-1,1]$ there replaced by $I$. In other words, the functions $(f(\cdot,y))_{y \in I}$ and $(-f(x,\cdot))_{x \in I}$ all satisfy the PL condition with the same constant.

\begin{proof}[Proof of Proposition~\ref{p.loc.conv}]
Using Proposition~\ref{p.simple} and the two-sided PL assumption, we see that $\dr_x^2 f(0,0) > 0$ and $\dr_y^2 f(0,0) < 0$. At the origin, the Jacobian of the vector field $(-\dr_x f, \dr_y f)$ is 
\begin{equation*}  
J := \begin{pmatrix}  
-\dr_x^2 f & -\dr_x \dr_y f \\
\dr_x \dr_y f  & \dr_y^2 f
\end{pmatrix}(0,0).
\end{equation*}
The symmetric part of this matrix, namely $(J + J^{\mathsf{T}})/2$, is negative definite. Hence, the eigenvalues of $J$ have negative real part, and the claim thus follows by the linear stability theorem. 
\end{proof}


\begin{proposition}[criterion for two-sided PL]
\label{p.crit.2PL}
Let $f : I^2 \to \R$ be a $\mcl C^2$ function. In order for the function $f$ to satisfy the two-sided PL condition, it suffices that for every $z \in I^2$, the following two implications hold:
\begin{equation}  
\label{e.impl1}
\partial_x f(z) = 0 \quad \implies \quad \partial_x^2 f(z) > 0, 
\end{equation}
\begin{equation}  
\label{e.impl2}
\partial_y f(z) = 0 \quad \implies \quad \partial_y^2 f(z) < 0.
\end{equation}
\end{proposition}
\begin{proof}
By symmetry, it suffices to establish the uniform PL condition with respect to the first variable. Consider the set 
\begin{equation}  
\label{e.def.mclZ}
\mcl Z := \{z \in I^2 \ : \ \dr_x f(z) = 0\}.
\end{equation}
Since the function $f$ is $\mcl C^2$, for every $z \in \mcl Z$, one can find a constant $c_z > 0$ and an open neighborhood $U_z$ of $z$ such that $\dr_x^2 f\ge c_z$ in $U_z$. Since the set $\mcl Z$ is compact, we can cover it with a finite number of such neighborhoods, and thus we can build an open set $U \subset I^2$ containing $\mcl Z$ and a constant $c > 0$ such that $\dr_x^2 f \ge c$ on $U$. Arguing as in the proof of Proposition~\ref{p.simple}, we see that the condition \eqref{e.PL} holds for all the partial functions $(f(\cdot, y))_{y \in I}$ inside $U$ with the constant $C = \sup \dr_x^2 f/(2c^2)$. On the complement $I^2 \setminus U$, the derivative $|\partial_x f|$ is bounded away from zero, and thus one can adjust the constant in the PL condition so that it holds everywhere.
\end{proof}

%
%
%
%
%
%

\section{Construction of the function}
\label{s.constr}

In order to construct a function $f$ satisfying the requirements of Theorem~\ref{t.main}, we will start by constructing its level lines, which will be described as the flow lines of a particular vector field $\v$. Before defining $\v$ itself, we need to set up a handful of properties of the polynomial
\begin{equation*}  
P(T) := T^3 + T^2 + T - \frac 1 3.
\end{equation*}
Substituting $U-1/3$ for $T$ and applying Cardano's formula, we find that this polynomial has a unique real root $\ga$ given by
\begin{equation*}  
\ga := \frac 1 3 \Ll( \sqrt[3]{6\sqrt{2} + 8} - \sqrt[3]{6\sqrt{2}-8} - 1 \Rr) \simeq 0.2531,
\end{equation*}
and that
\begin{equation*}  
P(T) = (T-\ga)\Ll(T^2 + aT + b\Rr),
\end{equation*}
where for convenience, we have set $a := \ga + 1 \simeq 1.2531$ and $b := 1/(3\ga) = \ga^2 + \ga + 1 \simeq 1.3171$.
We let $\phi : \R_+ \to [1,+\infty)$ denote a $\mcl C^\infty$ function that is constant equal to $1$ on $[0,1/2]$, is non-decreasing, and is twice the identity on $[1,+\infty)$. Concretely, we can for instance take
\begin{equation*}  
\phi(t) = \begin{cases} 
1 & \text{ if } t \leq \tfrac 1 2, \\ 
1 + \Ll(2t-1\Rr)\rho(2t-1) & \text{ if } t \in (\tfrac 1 2, 1), \\ 
2t & \text{ if }t \geq 1, 
\end{cases}
\end{equation*}
where $\rho : (0,1) \to (0,1)$ is given by $\rho(u) := e^{-1/u}\big/(e^{-1/u} + e^{-1/(1-u)})$.
With this in place, we define, for every $x,y \in \R$,
\begin{equation}  
\label{e.def.v}
\v(x,y) = 
\begin{pmatrix}  
\v_1(x,y) \\
\v_2(x,y)
\end{pmatrix}
:= 
\begin{pmatrix}  
(\ga y - x) \phi(x^2 + axy + b y^2) \\
(y + \ga x) \phi(y^2 - axy + b x^2)
\end{pmatrix}.
\end{equation}
Note that since the polynomial $T^2 + a T + b$ has no real root, the arguments inside the function $\phi$ in \eqref{e.def.v} are indeed nonnegative. The exact form of $\v$, with the factorization into the polynomial $P$ at large distances, will be used around \eqref{e.v.outside} to identify the large-scale integral of motion of the dynamics.

We say that a trajectory $(M(t))_{t \in I}$, where $I$ is an interval of $\R$ and $M(t) \in \R^2$, is a \emph{flow line of} $\v$ if for every $t \in I$, we have
\begin{equation*}  
\dr_t M(t) = \v(M(t)).
\end{equation*}
We say that it is a \emph{complete flow line of} $\v$  if it is a flow line of $\v$ and the interval $I$ is maximal for this property. We also introduce the notation
\begin{equation*}  
\mcl X := \{(x,y) \in \R^2 \ : \ x = \ga y \text { or } y = - \ga x\}.
\end{equation*}
\begin{proposition}
\label{p.flow}
Let $(M(t))_{t \in I}$ be a complete flow line of $\v$.
Exactly one of these four possibilities is valid.

(0) The flow line stays put at the origin for all times.

(1) The flow line intersects $\mcl X$ exactly once.

(2) The flow line never intersects $\mcl X$, and $\lim_{t \to +\infty} M(t) = 0$.

(3) The flow line never intersects $\mcl X$, and $\lim_{t \to -\infty} M(t) = 0$.

\end{proposition}
\begin{proof}
The origin is indeed a fixed point. For the other cases, we make use of the change of variables $\ell_1 := x - \ga y$ and $\ell_2 := y + \ga x$, so that $\mcl X = \{\ell_1 = 0\} \cup \{\ell_2 = 0\}$, and we also denote $\phi_1 := \phi(x^2 + axy + by^2) \ge 1$ and
$\phi_2 := \phi(y^2 - axy + bx^2) \ge 1$. We use the notation $M(t) = (x(t),y(t))$ and use an upper dot to indicate the time derivative along the flow line.  Using $\dot x = \v_1 = -\ell_1\phi_1$ and
$\dot y = \v_2 = \ell_2\phi_2$, we see that
\begin{equation*}
\dot \ell_1 = -\ell_1\phi_1 - \ga \ell_2 \phi_2, \qquad \dot \ell_2 = -\ga \ell_1 \phi_1 + \ell_2 \phi_2.
\end{equation*}
We introduce the four open quadrants $Q_{\pm \pm}$ defined by
\begin{equation*}  
Q_{\pm \pm} := \{\pm \ell_1 > 0 \quad \text{ and } \quad \pm \ell_2 > 0\}.
\end{equation*}

\smallskip

\noindent \emph{Step 1.} We first argue that the flow line can cross $\{\ell_1 = 0\} \cup \{\ell_2 = 0\}$ at most once. On the half-line
$\{\ell_1 = 0,\, \ell_2 > 0\}$, we have $\dot \ell_1 = -\ga \ell_2 \phi_2 < 0$, so every crossing there goes from $Q_{++}$ into $Q_{-+}$. On $\{\ell_2 = 0,\, \ell_1 > 0\}$, we have $\dot \ell_2 = -\ga \ell_1\phi_1 < 0$, so every crossing goes from $Q_{++}$ into $Q_{+-}$. Similarly, the two boundary segments $\{\ell_1=0,\,\ell_2<0\}$ and $\{\ell_2=0,\,\ell_1<0\}$ are crossed from $Q_{--}$
into $Q_{+-}$ and $Q_{-+}$ respectively. In particular, every trajectory that enters $Q_{-+}$ never leaves it, and similarly for $Q_{+-}$. Conversely, every backward-in-time trajectory that enters $Q_{++}$ or $Q_{--}$ never leaves it. Since every crossing of
$\{\ell_1=0\}\cup\{\ell_2=0\}$ away from the origin corresponds to a transition from
$\{Q_{++},Q_{--}\}$ into $\{Q_{-+},Q_{+-}\}$, at most one such crossing can occur along any
trajectory.

\smallskip

\noindent \emph{Step 2.} We now argue that if a flow line never intersects $\mcl X$, then it is
trapped in one of the four quadrants (these being the connected components of
$\R^2 \setminus \mcl X$), and converges to the origin as $t \to +\infty$ if the
quadrant is $Q_{++}$ or $Q_{--}$, and as $t \to -\infty$ if the quadrant is $Q_{+-}$
or $Q_{-+}$.

We first consider the case of a flow line trapped in $Q_{++}$. Shifting time, we
assume that $0 \in I$, and we set $T := \sup I$. Since $\ell_1, \ell_2 > 0$ and
$\phi_1, \phi_2 \ge 1$ on $Q_{++}$, we have
$\dot \ell_1 = -\ell_1 \phi_1 - \ga \ell_2 \phi_2 < 0$ along the flow line, so the
function $t \mapsto \ell_1(t)$ is decreasing on $[0,T)$, with values in
$(0, \ell_1(0)]$. Integrating the identity
$-\dot \ell_1 = \ell_1 \phi_1 + \ga \ell_2 \phi_2$, whose right-hand side is a sum of
two positive terms, we obtain that
\begin{equation}
\label{e.integrable}
\int_0^{T} \ell_1 \phi_1 \, \d t \le \ell_1(0)
\qquad \text{and} \qquad
\int_0^{T} \ell_2 \phi_2 \, \d t \le \ga^{-1} \ell_1(0).
\end{equation}
Since $\dot \ell_2 \le \ell_2 \phi_2$, the second inequality in \eqref{e.integrable}
yields, for every $t \in [0,T)$,
\begin{equation*}
\ell_2(t) \le \ell_2(0) + \int_0^{t} \ell_2 \phi_2 \, \d s
\le \ell_2(0) + \ga^{-1} \ell_1(0).
\end{equation*}
The flow line therefore remains in a compact subset of $\R^2$ on $[0,T)$. If $T$ were
finite, the solution could then be extended beyond $T$, contradicting the maximality
of $I$; hence $T = +\infty$. Recalling that $\phi_1, \phi_2 \ge 1$, we deduce from
\eqref{e.integrable} that $\ell_1$ and $\ell_2$ are integrable on $[0,+\infty)$. Being
also monotone, the function $\ell_1$ must converge to $0$ as $t \to +\infty$. The
function $\ell_2$ may fail to be monotone, but since $0 < \ga < 1$ (recall that
$\ga \simeq 0.2531$), we have
\begin{equation*}
\frac{\d}{\d t} \Ll(\ell_2 - \ell_1\Rr)
= (1-\ga) \ell_1 \phi_1 + (1+\ga) \ell_2 \phi_2 \ge 0,
\end{equation*}
so the function $\ell_2 - \ell_1$ is non-decreasing and bounded, and thus converges as
$t$ tends to infinity. Hence $\ell_2 = (\ell_2 - \ell_1) + \ell_1$ converges as well,
and since $\ell_2$ is integrable on $[0,+\infty)$, its limit must be $0$. We have thus
shown that flow lines trapped in $Q_{++}$ converge to the origin as $t \to +\infty$.

The remaining quadrants are handled by symmetry. The vector field $\v$ is odd, so the
mapping $z \mapsto -z$ maps complete flow lines to complete flow lines; since it maps
$Q_{--}$ onto $Q_{++}$, flow lines trapped in $Q_{--}$ also converge to the origin as
$t \to +\infty$. Moreover, denoting by $R(x,y) := (-y,x)$ the rotation by $\pi/2$, one
can check from \eqref{e.def.v} that $\v \circ R = R^{-1} \circ \v$, and thus that if
$(M(t))_{t \in I}$ is a complete flow line, then so is $(R M(-t))_{t \in -I}$. Since
$R$ maps $Q_{+-}$ onto $Q_{++}$ and $Q_{-+}$ onto $Q_{--}$, we conclude that flow
lines trapped in $Q_{+-}$ or $Q_{-+}$ converge to the origin as $t \to -\infty$.
\end{proof}
Since the vector field $\v$ is smooth, it is not possible for two flow lines to intersect. In order to define the function $f$, it thus suffices to prescribe its value on every complete flow line. Let $(M(t))_{t \in I}$ be a complete flow line. If this flow line never intersects $\mcl X$, or if it stays put at the origin, then we set $f(M(t)) = 0$ for every $t \in I$. If the flow line does intersect $\mcl X$, then by Proposition~\ref{p.flow}, there exists a unique time $t_0 \in I$ such that $M(t_0)$ belongs to~$\mcl X$. Writing $(x_0, y_0 ) = M(t_0)$, we set, for every $t \in I$,
\begin{equation*}  
f(M(t)) = \frac \ga 2 x_0^2 + x_0 y_0 - \frac \ga 2 y_0^2.
\end{equation*}
We note that the prescribed value is never zero in this case, since for every $x,y \in \R$, we have
\begin{equation}  
\label{e.f.on.mclX}
f(\ga y, y ) = \frac{1}{2} (\ga^3 + \ga)y^2 \quad \text{ and } \quad f(x, -\ga x) = -\frac 1 2 (\ga^3 + \ga) x^2.
\end{equation}
Let us pause to describe the geometric meaning of the set $\mcl X$, which is the union
of the two orange lines on Figure~\ref{f.vector_field}. Since the function $\phi$
takes only positive values, we see from \eqref{e.def.v} that $\v_1$ vanishes exactly
on the line $\{x = \ga y\}$, and $\v_2$ vanishes exactly on the line
$\{y = -\ga x\}$. In other words, the vector field $\v$, and thus also the level line
of $f$, is vertical on the first line and horizontal on the second (outside of the
origin, where $\v$ vanishes). As will be shown in Proposition~\ref{p.good.f}, the
function $f$ is differentiable with a gradient that vanishes only at the origin; and
since $f$ is constant along the flow lines of $\v$, its gradient is everywhere
orthogonal to $\v$. Outside of the origin, the condition $\dr_x f = 0$ is therefore
equivalent to $\nabla f$ being vertical, that is, to $\v$ being horizontal; and
similarly with the roles of ``horizontal'' and ``vertical'' exchanged. In summary,
\begin{equation}
\label{e.orange.lines}
\{\dr_x f = 0\} = \{y = -\ga x\}
\qquad \text{and} \qquad
\{\dr_y f = 0\} = \{x = \ga y\},
\end{equation}
and $\mcl X$ is the union of these two sets. This explains why the set $\mcl X$ is
central to the verification of the two-sided PL condition: by
Proposition~\ref{p.crit.2PL}, it suffices to check the implications \eqref{e.impl1}
and \eqref{e.impl2}, and by \eqref{e.orange.lines}, these are non-trivial only on
$\mcl X$. Prescribing the values of $f$ on $\mcl X$ as in \eqref{e.f.on.mclX} gives us
direct access to the behavior of $f$ near this set. We also point out that, once the
two-sided PL condition is established, the identities \eqref{e.orange.lines} imply
that the points of $\{y = -\ga x\}$ are minimizers of the partial functions
$f(\cdot,y)$, while the points of $\{x = \ga y\}$ are maximizers of the partial
functions $f(x,\cdot)$.

We recall that we denote by $B_r$ the open Euclidean ball of radius $r$ centered at the origin.
\begin{proposition}
\label{p.near.origin}
There exists $r > 0$ such that for every $(x,y) \in B_r$, we have
\begin{equation} 
\label{e.near.origin}
f(x,y) = \frac \ga 2 x^2 + xy - \frac \ga 2 y^2.
\end{equation}
\end{proposition}
\begin{proof}
We decompose the proof into two steps. 

\smallskip

\noindent \emph{Step 1.} In this step, we introduce some notation and a convenient change of coordinates. 
For every $x,y \in \R$, we write $g(x,y) := \frac \ga 2 x^2 + xy - \frac \ga 2 y^2$ and
\begin{equation}  
\label{e.def.w}
\mathbf{w}(x,y) := \begin{pmatrix} \ga y - x \\ y + \ga x \end{pmatrix}.
\end{equation}
Since the quadratic forms $x^2 + axy + by^2$ and $y^2 - axy + bx^2$ are positive definite
(recall that $T^2 + aT + b$ has no real root), there exists $r_0 > 0$ such that both are
at most $\tfrac{1}{2}$ on $B_{r_0}$. Since $\phi = 1$ on $[0,\tfrac{1}{2}]$, we have
$\v = \mathbf{w}$ on $B_{r_0}$.

We set $\mu := \sqrt{1+\ga^2}$ and introduce the linear coordinates
\[
  u := \ga x + (1+\mu)y, \qquad v := \ga x + (1-\mu)y.
\]
A direct computation shows that along the flow lines of $\mathbf{w}$, we have
$\dot u = \mu\, u$ and $\dot v = -\mu\, v$,
so that $uv$ is preserved along the flow. Moreover, for every $x,y \in \R$, we have
\begin{equation}
  \label{e.g.uv}
  g(x,y) = \frac{\ga}{2}x^2 + xy - \frac{\ga}{2}y^2 = \frac{1}{2\ga}\,uv.
\end{equation}
Setting $\kappa := (\mu+1)/(\mu-1)$, one can check that, in $(u,v)$ coordinates,
\[
  \mcl X = \{u = \kappa\, v\} \cup \{u = -v\}.
\]
In particular, the set $\mcl X$ meets every open quadrant of the $(u,v)$-plane: the line
$\{u=\kappa\, v\}$ enters the quadrants where $u$ and $v$ have the same sign, while  the line $\{u=-v\}$ enters the quadrants where $u$ and $v$ have opposite signs. We also note that, since the vectors $(\ga, 1+\mu)$ and $(\ga,1-\mu)$ are orthogonal, and their squared norms are $2\mu(\mu+1)$ and $2\mu(\mu-1)$ respectively, we have
\begin{equation}  
\label{e.ortho}
x^2 + y^2 = \frac{u^2}{2\mu(\mu+1)} + \frac{v^2}{2\mu(\mu-1)}.
\end{equation}

\smallskip

\noindent \emph{Step 2.} We now show that $f = g$ on $B_r$ for a suitable $r \in (0,r_0]$.

Consider first a point $p = (x,y) \in B_r$ with $g(x,y) \neq 0$, and denote its $(u,v)$ coordinates by $(u_0,v_0)$. We define $c := u_0 v_0$, which is non-zero by~\eqref{e.g.uv}.
The flow line of $\mathbf{w}$ through the point $p$ is an arc of the hyperbola $\{uv = c\}$.
Each branch of this hyperbola lies in a single open quadrant of the $(u,v)$-plane,  and crosses $\mcl X$ exactly
once, at a point which we denote by $q$. Indeed, if for instance we suppose that the branch of hyperbola being discussed is in the quadrant $\{u > 0, v > 0\}$, then the line $\{u = - v\}$ does not intersect this quadrant, while the line $\{u = \kappa v\}$ intersects it at the point $(\sqrt{\kappa c}, \sqrt{c/\kappa})$ and only there; for the quadrant $\{u < 0, v < 0\}$, the $(u,v)$-coordinates of the point $q$ are $(-\sqrt{\kappa c}, -\sqrt{c/\kappa})$; while for $\{\pm u > 0, \mp v > 0\}$ (i.e., when $c < 0$), the $(u,v)$-coordinates of $q$ are $(\pm \sqrt{|c|}, \mp \sqrt{|c|})$. In all cases, recalling \eqref{e.ortho} and that $\mu^2 = 1+\gamma^2$, we have 
\begin{equation*}  
|q|^2 = \frac{|c|}{\gamma^2}. 
\end{equation*}
On the other hand, the symmetric matrix associated with the quadratic form $g$ has eigenvalues $\pm \mu/2$. Using \eqref{e.g.uv}, we thus obtain
\begin{equation*}  
|c| = 2 \gamma |g(p)| \le \gamma \mu |p|^2. 
\end{equation*}
Combining the two previous displays yields that
\begin{equation*}  
|q|^2 \le \frac{\mu}{\ga} |p|^2. 
\end{equation*}
It remains to verify that the arc going from $p$ to $q$ stays in the region where $\v = \mathbf w$. We note that on a branch of $\{u v = c\}$, the quantity $u^2 + v^2 = u^2 + c^2/u^2$ is convex in $u^2$; so the arc joining $p$ to $q$ remains in any ``$(u,v)$-ball'' containing both $p$ and $q$. Using again \eqref{e.ortho}, it is thus clear that we can choose $r > 0$ sufficiently small that for every $p \in B_r$ and corresponding $q \in \mathcal{X}$, the arc of hyperbola connecting $p$ and $q$ stays within $B_{r_0}$.
Since $\v = \mathbf{w}$ on $B_{r_0}$, this part of a flow line of $\mathbf w$ is also part of a flow line of $\v$,
which thus intersects $\mcl X$. By definition, $f$ is constant along this flow line and
equals $g$ at the crossing point; since $g$ is also constant along this flow line
by~\eqref{e.g.uv}, we conclude that $f(x,y) = g(x,y)$.

It remains to consider the case when $p \in \{g=0\} \cap B_r$. The equality $f(0) = g(0) = 0$ follows directly from the definition of $f$. Besides the origin, the set $\{g = 0\}$ consists of the four half-axes $\{u = 0, \pm v > 0\}$ and $\{v = 0, \pm u > 0\}$. These four half-axes are the only flow lines of $\mathbf w$ that converge to the origin (as $t \to +\infty$ for the first pair, and as $t \to - \infty$ for the second pair). The flow lines of $\mathbf{v}$ that converge to the origin as $t$ tends to $\pm \infty$ from Proposition~\ref{p.flow}
must therefore, once restricted to $B_{r}$, coincide with the lines along which $g$ vanishes. By the definition of $f$, we thus have $f(p) = 0 = g(p)$, as desired. 
\end{proof}

\begin{proposition}
\label{p.good.f}
The function $f$ is $\mcl C^\infty$ on $\R^2$, has a unique critical point at the origin, and for every $R < +\infty$, it satisfies the two-sided PL condition on $[-R,R]^2$. 
\end{proposition}
\begin{proof}
By Proposition~\ref{p.near.origin}, there exists $r > 0$ such that the identity
\eqref{e.near.origin} holds for every $(x,y) \in B_r$. In particular, the restriction
of $f$ to $B_r$ is $\mcl C^\infty$; moreover, since the quadratic form on the right
side of \eqref{e.near.origin} is non-degenerate, the origin is the only point of $B_r$
at which the gradient of $f$ vanishes.

We next show that $f$ is $\mcl C^\infty$ with non-vanishing gradient in a neighborhood
of every point $z_0 \in \mcl X \setminus \{0\}$. The idea is to use the flow of $\v$ to
build local coordinates in which one coordinate moves along the flow lines, so that $f$
depends only on the other, transverse coordinate. For definiteness, we assume that
$z_0 = (x_0, -\ga x_0)$ with $x_0 \neq 0$, the case of the other line making up
$\mcl X$ being similar. We parametrize the line $\{y = -\ga x\}$ by
$\si(s) := z_0 + s (1+\ga^2)^{-1/2}(1,-\ga)$, and we denote by $(\Phi_t)$ the flow of
the vector field $\v$, so that $t \mapsto \Phi_t(z)$ is the flow line of $\v$ started
at $z$. Since $\v$ is $\mcl C^\infty$, the mapping $\Phi : (s,t) \mapsto \Phi_t(\si(s))$
is well-defined and $\mcl C^\infty$ in a neighborhood of $(0,0)$, and its differential
at $(0,0)$ maps the canonical basis of $\R^2$ to the pair of vectors
$(\si'(0), \v(z_0))$. On the line $\{y = -\ga x\}$, the coordinate $\v_2$ vanishes,
while $\v_1$ vanishes only at the origin, since the two lines making up $\mcl X$
intersect only there. The vector $\v(z_0)$ is therefore horizontal and non-zero, while
$\si'(0)$ is not horizontal, so the differential of $\Phi$ at $(0,0)$ is invertible. By
the inverse function theorem, the mapping $\Phi$ is a $\mcl C^\infty$ diffeomorphism
from a neighborhood of $(0,0)$ onto a neighborhood $V$ of $z_0$; we write
$z \mapsto (s(z), t(z))$ for its inverse. Every $z \in V$ belongs to the flow line of
$\v$ started at $\si(s(z))$, and by Proposition~\ref{p.flow}, this flow line intersects
$\mcl X$ only at $\si(s(z))$. Recalling that $f$ was defined to be constant along each
complete flow line, and using \eqref{e.f.on.mclX}, we obtain that for every $z \in V$,
\begin{equation}
\label{e.flowbox}
f(z) = f(\si(s(z))) = -\frac 1 2 (\ga^3+\ga) \Ll( x_0 + \frac{s(z)}{\sqrt{1+\ga^2}} \Rr)^2 .
\end{equation}
In the coordinates $(s,t)$, the function $f$ is thus constant in the flow
coordinate~$t$, and polynomial in the transverse coordinate $s$. Since
$z \mapsto s(z)$ is $\mcl C^\infty$ on $V$, the identity \eqref{e.flowbox} shows that
$f$ is $\mcl C^\infty$ on $V$. Taking gradients in \eqref{e.flowbox} and using that
$x_0 \neq 0$ and that $\nabla s$ never vanishes (the differential of $\Phi^{-1}$ being
invertible), we also see, after possibly shrinking $V$, that the gradient of $f$ does
not vanish on $V$.

Finally, let $z \in \R^2 \setminus \{0\}$ be arbitrary. By Proposition~\ref{p.flow},
the complete flow line of $\v$ passing through $z$ either intersects $\mcl X$, or
converges to the origin as $t \to +\infty$ or as $t \to -\infty$. In every case, there
exists $t_1 \in \R$ such that $z' := \Phi_{t_1}(z)$ belongs to
$(\mcl X \cup B_r) \setminus \{0\}$; notice that $z'$ cannot be the origin, since by
uniqueness of solutions, the only flow line passing through the origin is the constant
one. The mapping $\Phi_{t_1}$ is well-defined and $\mcl C^\infty$ on a neighborhood $U$
of $z$, and is a diffeomorphism onto its image, with inverse $\Phi_{-t_1}$. For every
$w \in U$, the points $w$ and $\Phi_{t_1}(w)$ lie on the same complete flow line, and
$f$ is constant along complete flow lines, so
\begin{equation*}
f = f \circ \Phi_{t_1} \qquad \text{on } U.
\end{equation*}
By the previous two paragraphs, the function $f$ is $\mcl C^\infty$ in a neighborhood
of~$z'$, so $f$ is also $\mcl C^\infty$ in a neighborhood of $z$. Taking gradients, we
get $\nabla f(z) = (D \Phi_{t_1}(z))^\mathsf{T} \, \nabla f(z')$, where
$D\Phi_{t_1}(z)$ is the (invertible) Jacobian matrix of $\Phi_{t_1}$ at $z$, and
$\nabla f(z') \neq 0$; hence $\nabla f(z) \neq 0$.

It remains to show that for every $R < +\infty$, the function $f$ satisfies the two-sided PL condition on $[-R,R]^2$. By Proposition~\ref{p.crit.2PL}, it suffices to verify that \eqref{e.impl1} and \eqref{e.impl2} are valid. These conditions are valid at the origin, by Proposition~\ref{p.near.origin}. Let $z = (x,y) \in \R^2 \setminus \{0\}$ be a point such that $\dr_x f(z) = 0$. Since $f$ has no critical point besides the origin, we must have that $\dr_y f(z) \neq 0$. Since $\nabla f$ must be orthogonal to $\v$, it follows that $\v_2(z) = 0$, that is, $y + \ga x = 0$ (and in particular $z \in \mcl X$). For definiteness, let us assume that $x > 0$. Using that $\dr_x f(z) = 0$ and \eqref{e.f.on.mclX}, we see that $\dr_y f(z) > 0$. Differentiating in $x$ the relation $\nabla f \cdot \v = 0$ and using that $\dr_x f(z) = 0$ and $\v_2(z) = 0$, we find that
\begin{equation}  
\label{e.id.dx2}
\Ll(\dr_x^2 f \, \v_1 + \dr_y f \, \dr_x \v_2\Rr)(z) = 0.
\end{equation}
Since $\v_2(x',y')$ is of the form $(y'+\ga x')$ times some function taking only positive values, we have that $\dr_x \v_2(z) > 0$. Since $x > 0$ and $y < 0$, we have $\v_1(x,y) < 0$. Recalling also that $\dr_y f(z) > 0$, we conclude from \eqref{e.id.dx2} that $\dr_x^2 f(z)$ must be strictly positive, as desired. The other cases can be handled similarly.
\end{proof}

We are now ready to prove the theorem.

\begin{proof}[Proof of Theorem~\ref{t.main}]
Up to a change of scale, we may as well show the statement of Theorem~\ref{t.main} with the domain $[-1,1]^2$ replaced by $[-R,R]^2$, for some $R < +\infty$ of our choosing. 
For $R < +\infty$ sufficiently large, we have for every $(x,y) \in \R^2 \setminus B_R$ that 
\begin{equation}  
\label{e.v.outside}
\v(x,y) = 
2 \begin{pmatrix}  
-y^3 P(x/y) \\
-x^3 P(-y/x)
\end{pmatrix}
= 
2 \begin{pmatrix}  
-x^3 - x^2 y - x y^2 + \frac{y^3}{3} \\
y^3 - x y^2 + x^2 y + \frac{x^3}{3}
\end{pmatrix}.
\end{equation}
For every $x,y \in \R$, we define
\begin{equation*}  
E(x,y) := 3 x^4 + 4 x^3 y + 6 x^2 y^2 - 4 x y^3 + 3 y^4.
\end{equation*}
On $\R^2 \setminus B_R$, we have
\begin{equation}  
\label{e.plot.twist}
\begin{pmatrix}  
\dr_x E \\
\dr_y E
\end{pmatrix}
= 6
\begin{pmatrix}  
-\v_1\\
\v_2
\end{pmatrix}.
\end{equation}
Let $(z(t))_{t \in I}$ be a GDA trajectory that stays in $\R^2 \setminus B_R$. By the definition in \eqref{e.def.grad.asc.desc}, we have
\begin{equation*}  
\dr_t (E(z(t))) =
\begin{pmatrix}  
\dr_x E \\
\dr_y E
\end{pmatrix}
\cdot
\begin{pmatrix}  
-\dr_x f \\
\dr_y f
\end{pmatrix}
(z(t)).
\end{equation*}
Using \eqref{e.plot.twist} and recalling that $\nabla f$ and $\v$ are orthogonal, we see that $E$ remains constant along the trajectory of $(z(t))_{t \in I}$. We now observe that every level line of $E$ is the boundary of an $L^4$ ball rotated by $\pi/8$. Indeed, this follows from the observation that, for $\al := \tan(\pi/8) = \sqrt{2} - 1$, we have
\begin{align*}  
& |x + \al y|^4 + |y - \al x|^4
\\
& \quad = (1+\al^4) x^4 + 4\al(1-\al^2) x^3 y + 12 \al^2 x^2 y^2 - 4 \al (1-\al^2)x y^3  + (1+\al^4) y^4
\\
& \quad = 2\al^2 E(x,y).
\end{align*}
It is therefore clear that the level lines of $E$ are closed curves that circle around the origin, and that some of those stay outside of $B_R$. We choose $R'$ such that $[-R',R']^2 \setminus B_R$ contains such level lines in its interior. A GDA flow starting from a point on one of these level lines will forever circle along it, since $E$ must remain constant along the GDA flow and the function $f$ has no critical point along this level line, by Proposition~\ref{p.good.f}. By this same proposition, the restriction of the function $f$ to $[-R',R']^2$ thus satisfies all the requirements to yield the validity of Theorem~\ref{t.main}.
\end{proof}

%
%
%
%
%
%

\section{From two-sided PL to strongly-convex/PL}
\label{s.convPL}

In this section, we prove Corollary~\ref{c.convPL}. 
We start by discussing the idea underlying the argument. The key observation is that, in order to verify that the function $f$ indeed satisfies all the properties announced in Theorem~\ref{t.main}, the knowledge of the level sets of $f$ is essentially sufficient. Hence, in order to prove Corollary~\ref{c.convPL}, we will look for a function that has the exact same level lines as $f$ (as displayed on Figure~\ref{f.vector_field}), but we will seek a new assignment of values to these level lines, hoping to speed up the growth of the function~$f$ towards large values so that it becomes convex in $x$. Concretely, this means that we build the function $F$ for Corollary~\ref{c.convPL} as $\Psi \circ f$, where $\Psi$ is an increasing function whose derivative grows sufficiently rapidly.

\begin{proof}[Proof of Corollary~\ref{c.convPL}]  
Let $f : \R^2 \to \R$ denote the function that was built in Section~\ref{s.constr}. For some $\be > 0$ to be determined, we define $F_\be : \R^2 \to \R$ by
\begin{equation*}  
F_\be := \be^{-1} e^{\be f}. 
\end{equation*}
The function $F_\be$ is $\mcl C^\infty$, with
\begin{equation}  
\label{e.dr1}
\nabla F_\be = \nabla f \,  e^{\be f}, 
\end{equation}
\begin{equation}  
\label{e.dr2}
\dr_x^2 F_\be = \Ll(\dr_x^2 f + \be (\dr_x f)^2\Rr) e^{\be f},
\end{equation}
and similarly with $x$ substituted by $y$. In particular, the function $F_\be$ has the same unique critical point as $f$. We recall from the proof of Proposition~\ref{p.good.f} that the function $f$ satisfies the implications \eqref{e.impl1} and \eqref{e.impl2}. From \eqref{e.dr1}-\eqref{e.dr2}, we see that the function $F_\be$ also satisfies these two implications. By Proposition~\ref{p.crit.2PL}, the function $F_\be$ thus satisfies the two-sided PL condition on every compact set. 

Recall that in the last paragraph of the previous section, we chose $R'$ sufficiently large to contain an open set of periodic orbits for the GDA flow on $f$; from now on, we fix $I := [-R',R']$ with this choice of $R'$. We now verify that, by choosing $\be$ sufficiently large, we can make sure that $F_\be$ is strongly convex in $x$ over~$I^2$. We let $\mcl Z$ denote the set defined in \eqref{e.def.mclZ}. Arguing as in the paragraph below~\eqref{e.def.mclZ}, we can find an open neighborhood $U$ of $\mcl Z$ and a constant $c > 0$ such that $\dr_x^2 f \ge c$ over $U$. By definition of $\mcl Z$, we have that $\dr_x f$ does not vanish in $I^2 \setminus U$, and by compactness, we thus have that 
\begin{equation*}  
c' := \inf_{I^2 \setminus U} (\dr_x f)^2 > 0.
\end{equation*}
Using that $f$ is a $\mcl C^2$ function, we can choose $\be$ sufficiently large that for every $z \in I^2$, we have
\begin{equation}
\label{e.}
\dr_x^2 f(z) + \be c'\ge 1.
\end{equation}
Letting
\begin{equation*}  
m_\be := \inf_{I^2} e^{\be f} > 0,
\end{equation*}
we thus have that
\begin{equation*}  
z \in U \quad \implies \quad \dr_x^2 F_\be(z) \ge c \, m_\be,
\end{equation*}
and
\begin{equation*}  
z \in I^2 \setminus U \quad \implies \quad \dr_x^2 F_\be(z) \ge m_\be,
\end{equation*}
as desired. 

It remains to verify that the trajectories of the GDA flow for $F_\be$ coincide with those of the GDA flow for $f$, up to a time reparametrization. This is immediate from \eqref{e.dr1}, since
\begin{equation*}  
(-\dr_x F_\be, \dr_y F_\be) = e^{\be f} (-\dr_x f, \dr_y f). 
\end{equation*}
All the desired properties of $F_\be$ are therefore verified on $I^2 = [-R',R']^2$. A change of scale completes the proof of Corollary~\ref{c.convPL}.
\end{proof}

\medskip

\noindent \textbf{Acknowledgements.} I would like to warmly thank Loucas Pillaud-Vivien for introducing me to this problem and for helpful feedback. I acknowledge the support of the ERC MSCA grant SLOHD (101203974), and of the French National Research Agency (ANR) under the France 2030 grant ANR-24-RRII-0002 operated by the Inria Quadrant Program. 

\small
\bibliographystyle{plain}
\bibliography{minmax}

\end{document}